\documentclass[12pt]{amsart}
\usepackage{amssymb}
\usepackage{tikz-cd}
\usepackage{amsrefs}
\usepackage{xfrac}
\usepackage{thmtools}
\usepackage{thm-restate}

\title[A sphere obstruction for the Hilbert shift quotient]{A sphere obstruction for the\\ Hilbert shift quotient}

\author {Avik Das}
\author {Christian Rosendal}
\address{Department of Mathematics\\University of Maryland\\4176 Campus Drive - William E. Kirwan Hall\\College Park, MD 20742-4015\\USA}
\email{adas1236@terpmail.umd.edu}
\email{rosendal@umd.edu}
\urladdr{sites.google.com/view/christian-rosendal/}

\newcommand{\norm}[1]{\lVert#1\rVert}

\newcommand{\forkindep}[1][]{\mathop{\mathop{\vcenter{\hbox{\oalign{\noalign{\kern-.3ex}
\hfil$\vert$\hfil\cr\noalign{\kern-.7ex}$\smile$\cr\noalign{\kern-.3ex}}}}}\displaylimits_{#1}}}

\newcommand{\conv}[2]{\mathop{\underset{#2}{\overset{#1}\longrightarrow}}}

\newcommand{\maps}[1]{\mathop{\overset{#1}\longrightarrow}}

\newcommand{\acts}[1]{\mathop{\overset{#1}\curvearrowright}}

\newcommand{\Mgd}[2]{\big\{  {#1}\;\big|\; {#2} \big\} }

\usetikzlibrary{decorations.markings}
\tikzset{negated/.style={
        decoration={markings,
            mark= at position 0.5 with {
                \node[transform shape] (tempnode) {$\backslash$};
            }
        },
        postaction={decorate}
    }
}

\newcommand {\Z}{\mathbb Z}

\newcommand {\R}{\mathbb R}
\newcommand {\C}{\mathbb C}
\newcommand {\T}{\mathbb T}

\newcommand{\eps}{\epsilon}

\newcommand {\ku} {\mathcal}

\theoremstyle{plain}
\newtheorem{thm}{Theorem}[section]
\newtheorem*{theorem*}{Theorem}
\newtheorem{cor}[thm]{Corollary}
\newtheorem{lemme}[thm]{Lemma}
\newtheorem{prop} [thm] {Proposition}

\theoremstyle{definition}

\newtheorem{rem}[thm]{Remark}

\definecolor{groen}{rgb}{0,0.5,.7}
\definecolor{gul}{rgb}{0.94,0.8,0}
\definecolor{blaa}{rgb}{0.16,0,0.6}
\definecolor{roed}{rgb}{1,0,0}

\makeindex

\newcommand{\ip}[2]{\langle #1,#2\rangle}
\newcommand{\Sphere}{\mathbb S}
\DeclareMathOperator{\supp}{supp}

\thanks{The authors' research was partially supported by the U.S. National Science Foundation under Grant Number DMS-2246986.}

\begin{document}

\begin{abstract}
Let $\mathcal H=\ell^2(\Z)$ and let $\Z$ act by bilateral shifts.
We prove that the quotient $\Sphere(\mathcal H)/\Z$ of the unit sphere,
equipped with its quotient metric, has infinite Hilbert distortion.
Consequently, $\mathcal H/\Z$ does not admit a bilipschitz embedding
into Hilbert space, answering a question of Cahill, Iverson, and Mixon.

The proof first embeds every flat torus $\R^n/L$, after rescaling,
into the quotient $\Sphere\big(\ell^2(\Z^n)\big)/\Z^n$ with distortion
arbitrarily close to $1$, where $\Z^n$ acts by translations.
We then show that every finite subset of this quotient embeds into
$\Sphere(\mathcal H)/\Z$ with distortion arbitrarily close to $1$.
The lower bounds of Khot and Naor for the Hilbert distortion of
flat tori then imply the result.
\end{abstract}

\maketitle


\section{Introduction}
A basic problem in bilipschitz invariant theory is to determine, for a real
Hilbert space $\mathcal H$ and a group $G\leqslant O(\mathcal H)$, whether the metric quotient
$\mathcal H/G$ admits a bilipschitz embedding into a Hilbert space. Cahill, Iverson, and
Mixon formulated this program in \cite{CahillIversonMixon}. Their motivation
comes from stable representations of data that are observed only modulo a group
of symmetries, including relabelings of graphs, permutations of point clouds,
rotations of landmarks, and time shifts of signals.

For the bilateral shift on $\ell^2(\Z;\R)$, they asked whether the quotient
$\ell^2(\Z;\R)/\Z$ admits a bilipschitz embedding into Hilbert space and, if so,
for its exact Euclidean distortion \cite{CahillIversonMixon}*{Problem 48}. Their
finite-dimensional computation gave the lower bound $1.3$. Subsequently,
Blum-Smith, Derksen, Mixon, Qaddura, and Vose noted that it remained open whether
the distortion of this quotient is infinite \cite{BlumSmithEtAl}*{Section 1.2.5}.
We prove that even the quotient of the unit sphere has infinite Hilbert
distortion.

For a metric space $(X,d_X)$, let $c_2(X)$ be the infimum of all $D\geqslant 1$
for which there are a real Hilbert space $\mathcal H$ and a map $X\maps\phi\mathcal H$ such that
\begin{equation}\label{eq:distortion-definition}
d_X(x,y)\leqslant \norm{\phi(x)-\phi(y)}
 \leqslant D d_X(x,y)
\end{equation}
for all $x,y\in X$. We set $c_2(X)=\infty$ when no such $D$ exists.

Set $\mathcal H=\ell^2(\Z;\R)$, let $\Sphere(\mathcal H)$ denote its unit sphere,  and let $\Z$ act by bilateral shifts on $\ku H$, that is, 
\[ 
(\tau_n f)(m)=f(m-n)
 \]
for $m,n\in \Z$. The quotient pseudometric on $\Sphere(\mathcal H)/\Z$ is a metric as shown in Section \ref{sec:background}.

\begin{thm}\label{thm:main}
For the bilateral shift action of $\Z$ on $\mathcal H=\ell^2(\Z;\R)$,
\[
        c_2\big(\Sphere(\mathcal H)/\Z\big)=\infty.
\]
In particular, neither $\Sphere(\mathcal H)/\Z$ nor $\mathcal H/\Z$
admits a bilipschitz embedding into Hilbert space.
\end{thm}

The proof yields two embedding results that may be of independent interest.
First, every flat $n$-torus admits, after rescaling, a bilipschitz embedding
into $\Sphere(\mathcal H_n)/\Z^n$, where $\mathcal H_n=\ell^2(\Z^n;\R)$,
with distortion arbitrarily close to $1$.
Theorem~\ref{thm:flat-tori-in-translation-quotient} constructs the
embedding directly from an equivariant family of real sequences. Their
coordinates are Fourier coefficients of functions supported near zero, and
Parseval's identity reduces the distance estimates to an elementary integral
estimate.

Second, every finite subset of $\Sphere(\mathcal H_n)/\Z^n$ embeds into
$\Sphere(\mathcal H_1)/\Z$ with distortion arbitrarily close to $1$.
For finitely supported representatives, a finite coding of coordinates
preserves the quotient distances exactly; approximation gives the general
case in Proposition~\ref{prop:finite-coding}. Consequently, every finite subset of
every flat torus can be represented, after rescaling and with arbitrarily
small bilipschitz error, inside the quotient by the bilateral shift.
Finite determination of Hilbert distortion and the lower bounds of Khot and
Naor for flat tori then imply Theorem~\ref{thm:main}.


\section{Translation quotients and flat tori}\label{sec:background}
All Hilbert spaces are real unless explicitly stated otherwise. For
$n\geqslant 1$, set
\[
        \mathcal H_n=\ell^2(\Z^n;\R),
\]
and let $\Z^n\acts\tau \mathcal H_n$ be the translation action defined by
\begin{equation}\label{eq:translation-action}
        (\tau_u f)(v)=f(v-u) 
\end{equation}
for $u,v\in \Z^n$. 
We note that if $(u_m)$ is an injective sequence in $\Z^n$ and $f\in \mathcal H_n$, then
$\tau_{u_m}f\conv{w}{m\to\infty}0$. It follows that every $\Z^n$-orbit in $\mathcal H_n$ is closed.
Therefore, the quotient pseudometric 
\begin{equation}\label{eq:sigma-metric}
 d_n([f],[g])=\inf_{u\in\Z^n}\norm{f-\tau_u g}
\end{equation}
on $\mathcal H_n/\Z^n$ is an actual metric and  ${\Sphere(\mathcal H_n)/\Z^n}$ a metric subspace of $\mathcal H_n/\Z^n$. Moreover, for
$f,g\in\Sphere(\mathcal H_n)$,
\begin{equation}\label{eq:sphere-correlation}
 d_n([f],[g])^2
 =2-2\sup_{u\in\Z^n}\ip{f}{\tau_u g}.
\end{equation}

We shall use the standard finite-determination principle
\begin{equation}\label{eq:finite-determination}
 c_2(X)=\sup\Mgd{c_2(F)}{F\subseteq X\text{ is finite}},
\end{equation}
which follows from a Hilbert-space ultraproduct argument; see
\cite{CahillIversonMixon}*{Proposition 31}.

A {\em full-rank lattice} in $\R^n$ is a subgroup of the form $L=A\Z^n$ for some
$A\in GL_n(\R)$. A \emph{flat torus} is a quotient $\R^n/L$, where $L\subseteq\R^n$ is a
full-rank lattice, equipped with the Euclidean quotient metric
\[
 d_L(x+L,y+L)=\min_{z\in L}\norm{x-y-z}_2.
\]
It is convenient to use coordinates in which the lattice is $\Z^n$. For this, set $\T^n=\R^n/\Z^n$ 
and let
\[
        \pi\colon\R^n\longrightarrow\T^n,
        \qquad \pi(x)=x+\Z^n,
\]
be the quotient map. For $A\in GL_n(\R)$, let $T_A$ denote the set $\T^n$
equipped with the metric
\begin{equation}\label{eq:flat-metric}
 \rho_A\big(\pi(x),\pi(y)\big)
 =\min_{u\in\Z^n}\norm{A(x-y-u)}_2
 \qquad(x,y\in\R^n).
\end{equation}
The value on the right is unchanged if either $x$ or $y$ is replaced by a
translate in $\Z^n$, so the formula is well defined. Moreover,
\[
        \pi(x)\longmapsto Ax+A\Z^n
\]
is an isometry from $T_A$ onto $\R^n/A\Z^n$. Thus every flat torus is isometric
to some $T_A$; the matrix $A$ determines its metric, not its underlying group.

The principal external input in the proof is the following consequence of the work
of Khot and Naor.

\begin{thm}[Khot--Naor]\label{thm:khot-naor}
There is a constant $c>0$ such that, for arbitrarily large $n$, one can find a
matrix $A\in GL_n(\R)$ satisfying
\begin{equation}\label{eq:khot-naor}
        c_2(T_A)\geqslant c\sqrt n.
\end{equation}
\end{thm}

This formulation follows from the existence of lattices $L\subseteq\R^n$
satisfying
\[
        c_2(\R^n/L)\geqslant c\sqrt n,
\]
since every full-rank lattice has the form $L=A\Z^n$. See Theorem 1.1 and the paragraph
following it in \cite{HavivRegev}, where the result is attributed to Khot and
Naor. The original argument appears in Theorem 5.1 and Corollary 5.2 of
\cite{KhotNaor}.

For comparison, Agarwal, Regev, and Tang \cite{AgarwalRegevTang} proved the
upper bound $c_2(T_A)=O(\sqrt{n\log n})$, uniformly over $A\in GL_n(\R)$.
Thus the lower bound in Theorem~\ref{thm:khot-naor} is sharp up to a factor
of order $\sqrt{\log n}$.


\section{Flat tori in translation quotients}

We begin with an estimate that will control distances in the translation
quotient. Its quadratic bounds near the origin will give the desired metric
comparison, while its positive lower bound away from the origin will prevent
distant lattice translates from making quotient distances too small.

\begin{lemme}\label{lem:cube-kernel}
Let $n\geqslant1$, put $Q=[-\sqrt3,\sqrt3]^n$, and define
$\Phi\colon\R^n\to\R$ by
\begin{equation}\label{eq:Phi-definition}
        \Phi(z)=\frac1{|Q|}\int_Q\cos\ip{z}{\xi}\,d\xi.
\end{equation}
For every $\eps>0$, there is a constant $c>0$ such that
\begin{equation}\label{eq:kernel-bounds}
 \min\left\{\frac{\norm{z}_2^2}{(1+\eps)^2},c\right\}
 \leqslant2(1-\Phi(z))\leqslant\norm{z}_2^2
 \qquad(z\in\R^n).
\end{equation}
\end{lemme}

\begin{proof}
Fix $\eps>0$. Since $|Q|=(2\sqrt3)^n$, for
$z=(z_1,\ldots,z_n)\in\R^n$ the definition of $\Phi$ gives
\begin{equation}\label{eq:Phi-iterated-integral}
 \Phi(z)=\frac1{(2\sqrt3)^n}
 \int_{-\sqrt3}^{\sqrt3}\!\cdots\!
 \int_{-\sqrt3}^{\sqrt3}
 \cos\left(\sum_{j=1}^n z_j\xi_j\right)
 \,d\xi_1\cdots d\xi_n.
\end{equation}
For $a,b\in\R$, the addition formula for cosine and the oddness of
$t\mapsto\sin(bt)$ give
\begin{align*}
 \frac1{2\sqrt3}\int_{-\sqrt3}^{\sqrt3}\cos(a+bt)\,dt
 &=\frac1{2\sqrt3}\int_{-\sqrt3}^{\sqrt3}
       \bigl(\cos a\cos(bt)-\sin a\sin(bt)\bigr)\,dt\\
 &=\frac{\cos a}{2\sqrt3}
       \int_{-\sqrt3}^{\sqrt3}\cos(bt)\,dt\\
 &=\cos a\,\frac{\sin(\sqrt3 b)}{\sqrt3 b},
\end{align*}
where the final factor $\frac{\sin(\sqrt3 b)}{\sqrt3 b}$ is interpreted as $1$ when $b=0$. Applying this identity
successively to the variables $\xi_n,\xi_{n-1},\ldots,\xi_1$ in
\eqref{eq:Phi-iterated-integral}, and using Fubini's theorem, yields
\begin{equation}\label{eq:Phi-product}
        \Phi(z)=\prod_{j=1}^n
        \frac{\sin(\sqrt3 z_j)}{\sqrt3 z_j},
\end{equation}
with the same convention when $z_j=0$. Since $|\sin t|<|t|$ for $t\neq0$,
we have $\Phi(z)<1$ for every $z\neq0$. Moreover,
$\Phi(z)\to0$ as $\norm{z}_2\to\infty$.

For $t\in\R$, define
\[
 \kappa(t)=
 \begin{cases}
 \dfrac{2(1-\cos t)}{t^2},&t\neq0,\\[6pt]
 1,&t=0.
 \end{cases}
\]
The function $\kappa$ is continuous at $0$. Choose $a>0$ such that
\begin{equation}\label{eq:kappa-lower}
        \kappa(t)\geqslant\frac1{(1+\eps)^2}
        \qquad(|t|\leqslant a),
\end{equation}
and put $\delta=a/\sqrt{3n}$. If $\norm{z}_2\leqslant\delta$ and
$\xi\in Q$, then the Cauchy--Schwarz inequality gives
\[
 |\ip{z}{\xi}|
 \leqslant\norm{z}_2\norm{\xi}_2
 \leqslant\delta\sqrt{3n}
 =a.
\]
Hence
\begin{align*}
 2(1-\Phi(z))
 &=\frac{2}{|Q|}\int_Q
       \bigl(1-\cos\ip{z}{\xi}\bigr)\,d\xi\\
 &=\frac1{|Q|}\int_Q
       \kappa(\ip{z}{\xi})\ip{z}{\xi}^2\,d\xi\\
 &\geqslant\frac1{(1+\eps)^2|Q|}
       \int_Q\ip{z}{\xi}^2\,d\xi
 =\frac{\norm{z}_2^2}{(1+\eps)^2}.
\end{align*}
The last equality follows from
\begin{equation}\label{eq:cube-second-moments}
 \frac1{|Q|}\int_Q\xi_j\xi_k\,d\xi
 =\begin{cases}1,&j=k,\\0,&j\neq k.\end{cases}
\end{equation}
For every $t\in\R$, the inequality $|\sin s|\leqslant|s|$ gives
\[
        2(1-\cos t)=4\sin^2(t/2)\leqslant t^2.
\]
Thus, by \eqref{eq:cube-second-moments},
\begin{equation}\label{eq:kernel-upper}
\begin{aligned}
 2(1-\Phi(z))
 &=\frac1{|Q|}\int_Q 2\bigl(1-\cos\ip{z}{\xi}\bigr)\,d\xi\\
 &\leqslant\frac1{|Q|}\int_Q\ip{z}{\xi}^2\,d\xi\\
 &=\norm{z}_2^2
\end{aligned}
\qquad(z\in\R^n).
\end{equation}

Since $\Phi(z)<1$ for $z\neq0$ and $\Phi(z)\to0$ at infinity, there is a
constant $c>0$ such that
\begin{equation}\label{eq:kernel-away-zero}
        2(1-\Phi(z))\geqslant c
        \qquad(\norm{z}_2\geqslant\delta).
\end{equation}
Together, the preceding estimates prove \eqref{eq:kernel-bounds}.
\end{proof}

We now construct an equivariant map into $\Sphere(\mathcal H_n)$ whose
induced map on the quotient gives the desired embedding of $T_A$.

\begin{thm}\label{thm:flat-tori-in-translation-quotient}
Let $A\in GL_n(\R)$ and let $\eps>0$. There are a $\Z^n$-equivariant map
$\R^n\maps F\Sphere(\mathcal H_n)$ and a number $\lambda>0$ such that the
induced map
\[
 p\colon T_A\longrightarrow\Sphere(\mathcal H_n)/\Z^n,
 \qquad p(\pi(x))=[F(x)],
\]
satisfies
\begin{equation}\label{eq:flat-torus-embedding}
 \lambda\rho_A(x,y)\leqslant d_n(p(x),p(y))
 \leqslant(1+\eps)\lambda\rho_A(x,y)
\end{equation}
for all $x,y\in T_A$. Here $\Z^n$ acts by translation on both $\R^n$ and
$\mathcal H_n$.
\end{thm}

\begin{proof}
Let $Q$ and $\Phi$ be as in Lemma~\ref{lem:cube-kernel}, and choose
$c>0$ such that \eqref{eq:kernel-bounds} holds for the given $\eps$.
Let $d=\operatorname{diam}(T_A)$, which is finite since $T_A$ is compact.
Put $C=[-\tfrac12,\tfrac12)^n$ and choose $R>0$ sufficiently large that
\begin{equation}\label{eq:R-choice}
 \frac1{2\pi R}A^TQ\subseteq C
 \qquad\text{and}\qquad
 \frac{d^2}{(1+\eps)^2R^2}\leqslant c,
\end{equation}
and put
\[
        P=\frac1{2\pi R}A^TQ.
\]
For $x\in\R^n$ and $v\in\Z^n$, define
\begin{equation}\label{eq:lift-coordinates}
 F(x)(v)=|P|^{-1/2}\int_P e^{2\pi i\ip{v-x}{\xi}}\,d\xi,
\end{equation}
where $|P|$ is the Lebesgue measure of $P$. Since $P=-P$, the sine term below 
integrates to zero, and hence
\begin{align*}
 F(x)(v)
 &=|P|^{-1/2}\Bigl(
   \int_P\cos\bigl(2\pi\ip{v-x}{\xi}\bigr)\,d\xi
   +i\int_P\sin\bigl(2\pi\ip{v-x}{\xi}\bigr)\,d\xi
   \Bigr)\\
 &=|P|^{-1/2}\int_P\cos\bigl(2\pi\ip{v-x}{\xi}\bigr)\,d\xi
 \in\R.
\end{align*}

We use Parseval's identity to compute the norm of the real sequence $F(x)$.
For this, equip $L^2(C;\C)$ with Lebesgue measure and the inner product
\[
 \ip{h}{k}_{L^2(C)}=\int_C h(\xi)\overline{k(\xi)}\,d\xi.
\]
The standard Fourier basis of this space is
\[
 e_v(\xi)=e^{-2\pi i\ip{v}{\xi}},\qquad v\in\Z^n.
\]
These functions form a complete orthonormal system; no normalizing factor
is needed since $|C|=1$. For each $x\in\R^n$, define the auxiliary function
\[
 h_x(\xi)=|P|^{-1/2}\mathbf1_P(\xi)e^{-2\pi i\ip{x}{\xi}},
 \qquad \xi\in C.
\]
Since $P\subseteq C$, this is an $L^2$ function with
$\int_C|h_x(\xi)|^2\,d\xi=|P|^{-1}|P|=1$.
Its Fourier coefficient at $e_v$ is exactly
\[
 \ip{h_x}{e_v}_{L^2(C)}
 =|P|^{-1/2}\int_P
   e^{-2\pi i\ip{x}{\xi}}e^{2\pi i\ip{v}{\xi}}\,d\xi
 =F(x)(v).
\]
Parseval's identity therefore gives
\[
 \sum_{v\in\Z^n}|F(x)(v)|^2
 =\sum_{v\in\Z^n}|\ip{h_x}{e_v}_{L^2(C)}|^2
 =\int_C|h_x(\xi)|^2\,d\xi=1.
\]
Thus we have a map
\[
        F\colon\R^n\longrightarrow\Sphere(\mathcal H_n).
\]
For every $u,v\in\Z^n$, the defining formula gives
\[
        F(x+u)(v)=F(x)(v-u).
\]
Consequently,
\begin{equation}\label{eq:lift-equivariance}
        F(x+u)=\tau_uF(x),
\end{equation}
so $F$ is equivariant.

For $x,y\in\R^n$, linearity gives
\[
 \ip{h_x-h_y}{e_v}_{L^2(C)}=F(x)(v)-F(y)(v).
\]
Thus Parseval's identity applied to $h_x-h_y$ gives
\begin{align*}
 \norm{F(x)-F(y)}_2^2
 &=\sum_{v\in\Z^n}|F(x)(v)-F(y)(v)|^2\\
 &=\int_C|h_x(\xi)-h_y(\xi)|^2\,d\xi\\
 &=\frac1{|P|}\int_P
   \left|e^{-2\pi i\ip{x}{\xi}}-e^{-2\pi i\ip{y}{\xi}}\right|^2\,d\xi\\
 &=\frac2{|P|}\int_P
   \bigl(1-\cos(2\pi\ip{x-y}{\xi})\bigr)\,d\xi.
\end{align*}
The last equality uses $|e^{ia}-e^{ib}|^2=2(1-\cos(a-b))$.
Under the change of variables $\xi=(2\pi R)^{-1}A^Tz$, we have
$d\xi/|P|=dz/|Q|$ and
$2\pi\ip{x-y}{\xi}=\ip{A(x-y)/R}{z}$. Hence
\begin{equation}\label{eq:lift-norm-distance}
 \norm{F(x)-F(y)}_2^2
 =2\left(1-\Phi\left(\frac{A(x-y)}R\right)\right).
\end{equation}

By equivariance, $F$ induces a well-defined map
\[
 p\colon T_A\longrightarrow\Sphere(\mathcal H_n)/\Z^n,
 \qquad p(\pi(x))=[F(x)].
\]
Fix $x,y\in T_A$ and choose lifts $\widetilde x,\widetilde y\in\R^n$
with $\pi(\widetilde x)=x$ and $\pi(\widetilde y)=y$.
Equations \eqref{eq:sigma-metric}, \eqref{eq:lift-equivariance}, and
\eqref{eq:lift-norm-distance} give
\begin{equation}\label{eq:quotient-distance}
\begin{aligned}
 d_n(p(x),p(y))^2
 &=\inf_{u\in\Z^n}\norm{F(\widetilde x)-F(\widetilde y+u)}_2^2\\
 &=\inf_{u\in\Z^n}
 2\left(1-\Phi\left(
       \frac{A(\widetilde x-\widetilde y-u)}R\right)\right).
\end{aligned}
\end{equation}
For every $u\in\Z^n$,
\[
 \norm{A(\widetilde x-\widetilde y-u)}_2\geqslant\rho_A(x,y).
\]
Since $\rho_A(x,y)\leqslant d$, the lower estimate in
\eqref{eq:kernel-bounds} and the choice \eqref{eq:R-choice} show that every
term in the infimum in \eqref{eq:quotient-distance} is at least
\[
        \frac{\rho_A(x,y)^2}{(1+\eps)^2R^2}.
\]
On the other hand, choose $u_0\in\Z^n$ realizing the minimum in
\eqref{eq:flat-metric}. The upper estimate in \eqref{eq:kernel-bounds} gives
\[
 d_n(p(x),p(y))^2\leqslant\frac{\rho_A(x,y)^2}{R^2}.
\]
Consequently,
\[
 \frac1{(1+\eps)R}\rho_A(x,y)
 \leqslant d_n(p(x),p(y))
 \leqslant\frac1R\rho_A(x,y).
\]
Taking $\lambda=((1+\eps)R)^{-1}$ proves the theorem.
\end{proof}


\section{Reduction to the bilateral shift}\label{sec:bilateral-shift}
The next lemma shows that every finite subset of a higher-rank translation
quotient can be represented, with arbitrarily small bilipschitz error, in the
quotient by the bilateral shift.

\begin{prop}\label{prop:finite-coding}
Let $n\geqslant 1$, let $F\subseteq{\Sphere(\mathcal H_n)/\Z^n}$ be finite, and let $\eps>0$.
There are a map $F\maps q\Sphere(\mathcal H_1)/\Z$ and a number $\lambda>0$ such that
\begin{equation}\label{eq:finite-coding-distortion}
 \lambda d_n(x,y)\leqslant d_1(q(x),q(y))
 \leqslant(1+\eps)\lambda d_n(x,y)
\end{equation}
for all $x,y\in F$.
\end{prop}

\begin{proof}
Finitely supported unit vectors are dense in $\Sphere(\mathcal H_n)$.
Moreover, for $f,g,f',g'\in\Sphere(\mathcal H_n)$,
\[
 \left|d_n([f],[g])-d_n([f'],[g'])\right|
 \leqslant \norm{f-f'}+\norm{g-g'}.
\]
Since $F$ is finite, its nonzero pairwise distances have a positive
lower bound. Thus sufficiently close approximations of its
representatives by finitely supported unit vectors change all
nonzero quotient distances by arbitrarily small relative amounts.
It therefore suffices to prove that these distances can be preserved
exactly when every point of $F$ has a finitely supported unit-vector
representative. Enumerate
\[
        F=\{[f_1],\ldots,[f_m]\}
\]
and choose a finite set $K\subseteq\Z^n$ containing $0$ and the support of
each $f_i$. Put
\[
        D=K-K.
\]
Choose integers $r\geqslant1$ and $s>2r$ such that
$D\subseteq[-r,r]^n$, and define a homomorphism
$\varphi\colon\Z^n\to\Z$ by
\begin{equation}\label{eq:mixed-radix-map}
        \varphi(u_1,\ldots,u_n)=\sum_{j=1}^n s^{j-1}u_j.
\end{equation}

We claim that $\varphi$ is injective on $D$. Suppose that $u,v\in D$ and
$\varphi(u)=\varphi(v)$. Setting $w=u-v$, we have
\[
        \sum_{j=1}^n s^{j-1}w_j=0
        \qquad\text{and}\qquad |w_j|\leqslant2r<s.
\]
Reduction modulo $s$ gives $w_1=0$. After dividing by $s$, the same argument
gives $w_2=0$, and repetition gives $w=0$. Thus $u=v$. Since $0\in K$,
we have $K\subseteq D$, so $\varphi$ is also injective on $K$.

Let $\ell^2(K;\R)$ be the subspace of $\mathcal H_n$ consisting of functions
supported on $K$. Define a real-linear isometry
\[
        J\colon\ell^2(K;\R)\longrightarrow\mathcal H_1
\]
by
\begin{equation}\label{eq:coding-isometry}
 (Jf)(\varphi(u))=f(u)\quad(u\in K),
 \qquad (Jf)(t)=0\quad(t\notin\varphi[K]).
\end{equation}
Fix $1\leqslant i,j\leqslant m$. If $t\notin\varphi[D]$, then
$\supp(Jf_i)$ and $\supp(\tau_tJf_j)$ are disjoint, and hence
\begin{equation}\label{eq:coded-zero-correlation}
        \ip{Jf_i}{\tau_tJf_j}=0.
\end{equation}
If $w\in D$, then
\begin{align*}
 \ip{Jf_i}{\tau_{\varphi(w)}Jf_j}
 &=\sum_{a\in\Z}(Jf_i)(a)(Jf_j)(a-\varphi(w))\\
 &=\sum_{\substack{u,v\in K\\
          \varphi(u)-\varphi(v)=\varphi(w)}}f_i(u)f_j(v).
\end{align*}
For $u,v\in K$, the difference $u-v$ belongs to $D$. Since $\varphi$ is
injective on $D$, the equality
$\varphi(u)-\varphi(v)=\varphi(w)$ is equivalent to $u-v=w$. Therefore
\begin{equation}\label{eq:coded-correlations}
 \ip{Jf_i}{\tau_{\varphi(w)}Jf_j}
 =\sum_{u\in\Z^n}f_i(u)f_j(u-w)
 =\ip{f_i}{\tau_wf_j}.
\end{equation}
If $w\notin D$, then $\supp(f_i)$ and $\supp(\tau_wf_j)$ are disjoint, so
$\ip{f_i}{\tau_wf_j}=0$. Equations
\eqref{eq:coded-zero-correlation} and \eqref{eq:coded-correlations} now imply
\[
 \sup_{t\in\Z}\ip{Jf_i}{\tau_tJf_j}
 =\max\left\{0,\max_{w\in D}\ip{f_i}{\tau_wf_j}\right\}
 =\sup_{w\in\Z^n}\ip{f_i}{\tau_wf_j}.
\]
It follows from \eqref{eq:sphere-correlation} that
\begin{equation}\label{eq:exact-coded-distances}
 d_1([Jf_i],[Jf_j])
 =d_n([f_i],[f_j])
 \qquad(1\leqslant i,j\leqslant m).
\end{equation}
Thus the assertion holds with $\lambda=1$ in the finitely supported case.
\end{proof}

\begin{cor}\label{cor:flat-tori-in-sigma-one}
Let $A\in GL_n(\R)$, let $E\subseteq T_A$ be finite, and let $\eps>0$. There
are a map $E\maps q\Sphere(\mathcal H_1)/\Z$ and a number $\lambda>0$ such that
\[
 \lambda\rho_A(x,y)\leqslant d_1(q(x),q(y))
 \leqslant (1+\eps)\lambda\rho_A(x,y)
\]
for all $x,y\in E$.
\end{cor}

\begin{proof}[Proof of Theorem \ref{thm:main}]
Suppose, towards a contradiction, that
\[
        D=c_2\big(\Sphere(\mathcal H_1)/\Z\big)<\infty.
\]
Let $A\in GL_n(\R)$ and let $E\subseteq T_A$ be finite. For every $\eps>0$,
Corollary \ref{cor:flat-tori-in-sigma-one} represents $E$, after rescaling, by
a finite subset $F\subseteq\Sphere(\mathcal H_1)/\Z$ with distortion at most $1+\eps$. Hence
\[
        c_2(E)\leqslant (1+\eps)c_2(F)
        \leqslant (1+\eps)D.
\]
Since $\eps>0$ is arbitrary, $c_2(E)\leqslant D$. Taking the supremum over all
finite $E\subseteq T_A$ and using \eqref{eq:finite-determination}, we obtain
\[
        c_2(T_A)\leqslant D
\]
for every $n\geqslant 1$ and every $A\in GL_n(\R)$. This contradicts Theorem
\ref{thm:khot-naor}, which gives such tori with arbitrarily large Hilbert
distortion. Therefore $c_2(\Sphere(\mathcal H_1)/\Z)=\infty$.
\end{proof}

\begin{rem}\label{rem:other-translation-quotients}
Theorem~\ref{thm:main} also gives nonembeddability results for other
translation quotients. By \cite{CahillIversonMixon}*{Theorem 44 and Example 47},
for every $d\geqslant1$, the Hilbert distortion of each of the spaces
\[
 L^2(\R^d;\R)/\R^d
 \qquad\text{and}\qquad
 L^2(\T^d;\R)/\T^d
\]
is bounded below by $c_2\big(\mathcal H/\Z\big)$ and is therefore
infinite. Here each group acts by translation, and each orbit space is
equipped with its quotient metric.
\end{rem}


\section*{AI statement and Lean verification}
An original version of the argument was found in May 2026 using OpenAI's
ChatGPT~5.5. The present manuscript represents a substantially reworked version of the original argument 
prepared by the authors with the aid of later versions of ChatGPT. The authors are fully responsible for
the content of the paper.

A Lean formalized proof of Theorem~\ref{thm:main} is available at 
\begin{center}
 \url{https://github.com/adas1236/SphereObstructionHilbertShiftQuotient}
 \end{center} 
The file \texttt{SphereObstructionHilbertShiftQuotient.lean} contains the formalizations of Theorem~\ref{thm:main}. While \texttt{SphereObstructionHilbertShiftQuotient.lean} aims to have clean definitions and theorem statements, there is no guarantee on the quality of other Lean files within the repository.
The formalization utilized OpenAI's Codex agent and GPT-5.5 Pro throughout the formalization process. The authors have checked the Lean statements and verified they accurately represent the content within this paper.

\enlargethispage{4\baselineskip}

\end{document}